\documentclass[a4paper,draft]{amsproc}

\usepackage{ amsmath, amsthm,  amscd, amsfonts, amssymb, graphicx, stix, mathtools, stmaryrd, color, xcolor, upgreek}
\usepackage[hyphens]{url} 

\newtheorem{theorem}{Theorem}[section]
\newtheorem{lemma}[theorem]{Lemma}
\newtheorem{proposition}[theorem]{Proposition}

\theoremstyle{definition}
\newtheorem{definition}[theorem]{Definition}

\theoremstyle{remark}
\newtheorem{remark}[theorem]{Remark}
\numberwithin{equation}{section}

\renewcommand{\leq}{\leqslant}

\renewcommand{\setminus}{\smallsetminus}
\title[ 2-local $ (\upvarphi, \uppsi) $-Derivations on Finite von Neumann Algebras]
 { 2-local $ (\upvarphi, \uppsi) $-Derivations on Finite  von Neumann Algebras}

\author[M. H. Fard]{Meysam Habibzadeh Fard }

\address{%
Department of Mathematics, \\
University of Guilan, \\
 Guilan, \\
  Iran }

\email{mhabibzadeh@phd.guilan.ac.ir,  meysam.habibzadeh@icloud.com}


\date{\today}

\begin{document}

{\begin{flushleft}\baselineskip9pt\scriptsize
MANUSCRIPT
\end{flushleft}}
\vspace{18mm} \setcounter{page}{1} \thispagestyle{empty}

\begin{abstract}
In this paper,  I introduce the concept of $ (\upvarphi, \uppsi) $-finite von Neumann algebras and I show that if $ \mathscr{M} $ is a finite and $ (\upvarphi, \uppsi) $-finite von Neumann algebra togather with condition $ \{ \big( \Delta(u+v)-\Delta(u)-\Delta(v)\big)^{*}\} \subseteq \uppsi(\mathscr{M}) $, then  each (approximately) 2-local $ (\upvarphi, \uppsi) $-derivation $ \delta $  on  $ \mathscr{M} $,   is a $ (\upvarphi, \uppsi) $-derivation.
\end{abstract}

\maketitle


\section{Introduction}     

Let $ \mathscr{A} $ be a commutative Banach algebra and $ \Phi_{\mathscr{A}} $ its spectrum. Then
for each $ \upvarphi, \uppsi \in \Phi_{\mathscr{A}}\cup \{0\} $, $ \mathbb{K} $ is an $ \mathscr{A} $-bimodule with the following actions
$$
a.z := \upvarphi(a)z, \hspace{.5cm} z.a := \uppsi(a)z, \hspace{.5cm}(a \in \mathscr{A}, z \in \mathbb{K})
$$
this module is denoted by $ \mathbb{K}_{\upvarphi, \uppsi} $ and we write $ \mathbb{K}_{\upvarphi} $ for $ \mathbb{K}_{\upvarphi, \upvarphi} $. Each one-dimentional
$ \mathscr{A} $-bimodule has the form $ \mathbb{K}_{\upvarphi, \uppsi} $ for some $ \upvarphi, \uppsi \in \Phi_{\mathscr{A}}\cup \{0\} $.
Let $ \mathscr{A} $ be an algebra, and let $ E $ he an $ \mathscr{A} $-bimodule. Then $ E $ is symmetric (or
commutative) if
$$
a . x = x . a, \hspace{.5cm}(a \in \mathscr{A}, x \in E) .
$$
For example, the $ \mathscr{A} $-himodule $ \mathbb{K}_{\upvarphi, \uppsi} $ is symmetric if and only if $ \upvarphi= \uppsi $.

A linear functional $ d $ on $ \mathscr{A} $  is a point derivation at $ \upvarphi $ if
$$
d(ab) = \upvarphi(a)d(b) + \upvarphi(b)d(a), \hspace{.5cm}(a, b \in \mathscr{A}).
$$
The definition of a $ (\upvarphi, \uppsi) $-derivation is derived from the definition of a point derivation (see \cite{Dales},  Proposition 1.8.10).

It is known that each derivation on a commutative and semi-simple Banach algebras  (commutative $ C^{*} $-algebras) is zero (\cite{Sakai2},  Corollary 2.2.3 and Corollary 2.2.8), but we can study  $ (\upvarphi, \uppsi) $-derivations on such algebras.

In Section 2, Theorem \ref{1}, is an extention of (\cite{Dales}, Theorem 2.8.63). Indeed, we show that if $ \mathscr{A} $ is a weakly amenable commutative Banach algebras, $ \upvarphi \in Hom(\mathscr{A}) $ is onto and $ X $ is a Banach  $ \mathscr{A} $-bimodule with the property that "$ a.x\neq 0 \Leftrightarrow x.a\neq 0 $"~$(a\in \mathscr{A}, x\in X) $, then each $ \upvarphi $-derivation from $ \mathscr{A} $ to $ X $ is zero.
 In  Theorem \ref{2}, we prove that each $ C^{*} $-algebra which has a separating family of normal tracial states is commutative.

In Section 3,  we  survey about Johnson's theorem for Jordan $ (\upvarphi, \uppsi) $-derivations on abelian $ C^{*} $-algebras.  Indeed, we show that   every bounded Jordan $ (\upvarphi, \uppsi) $-derivation  from a commutative $ C^{*} $-algebra $ \mathscr{A} $  into a Banach $ \mathscr{B}-\mathscr{C} $-bimodule $ X $ is a $ (\upvarphi, \uppsi) $-derivation. 

In Section 4,  Theorem \ref{44}, we prove that each $ (\upvarphi, \uppsi) $-derivation on finite von Neumann algebra $ \mathscr{M} $, is $ (\upvarphi, \uppsi) $-inner, this is an extention of (\cite{Sakai}, Theorem 4.1.6).  We introduce the concept of $ (\upvarphi, \uppsi) $-finite von Neumann algebras and  we use  Theorem \ref{3}, to prove that  every (approximately) 2-local $ (\upvarphi, \uppsi) $-derivation  on  finite and $ (\upvarphi, \uppsi) $-finite von Neumann algebras with condition $ \{ \big( \Delta(u+v)-\Delta(u)-\Delta(v)\big)^{*} \} \subseteq \uppsi(\mathscr{M}) $,  is a $ (\upvarphi, \uppsi) $-derivation.


\section{$ (\upvarphi, \uppsi) $-derivations}

Suppose that $ \mathscr{A}, \mathscr{B}, \mathscr{C} $ are Banach algebras and let
$ \upvarphi \in Hom( \mathscr{A}, \mathscr{B}) $,
 $ \uppsi \in Hom( \mathscr{A}, \mathscr{C}) $
 \footnote{
 If $ \mathscr{A}, \mathscr{B} $ are Banach algebras and $ \upvarphi \in Hom( \mathscr{A}, \mathscr{B}) $, then
  $ \mathscr{B} $
  is an $ \mathscr{A} $-bimodule with actions
\begin{equation*}
 a~\smallblacktriangleright ~b:= \upvarphi(a)b, \hspace{.25cm}
 b~ \smallblacktriangleleft~a:= b\upvarphi(a), \hspace{.5cm}(a\in \mathscr{A}, b\in \mathscr{B})
\end{equation*}
 }
 and let $ E $ be a Banach $ \mathscr{B}-\mathscr{C} $-module
 \footnote{If $ \mathscr{B} $, $ \mathscr{C} $ are Banach algebras, then $ E $ is a Banach $ \mathscr{B}-\mathscr{C} $-module, if it is a left Banach $ \mathscr{B} $-module and right Banach $ \mathscr{C} $-module. Clearly it   satisfy the following condition
 $$
 \|b~ \cdot~x~ \cdot~c\|_{E} \leq \|b\|_{\mathscr{B}}\|x\|_{E}\|c\|_{\mathscr{C}}
 $$.}.

\begin{definition}
 A linear operator $ \delta :  \mathscr{A}\longrightarrow E $ is a  $ (\upvarphi, \uppsi) $-derivation if it satisfies
$ \delta(ab)= \delta(a)\cdot \uppsi(b) + \upvarphi(a)\cdot \delta(b) $, ($ a,b \in \mathscr{A} $). A $ (\upvarphi, \uppsi) $-derivation $ \delta $ is called $ (\upvarphi, \uppsi) $-inner derivation if there exists $ x \in E $ such that $ \delta(a) = x\cdot \uppsi(a) - \upvarphi(a)\cdot x $
  ($ a \in \mathscr{A} $).
\end{definition}

Let $ \mathcal{Z}_{\upvarphi, \uppsi}^{1}(\mathscr{A}, E) $, denote the set of all continuous $ (\upvarphi, \uppsi) $-derivations from $ \mathscr{A} $ to $ E $ and let $ \mathcal{N}_{\upvarphi, \uppsi}(\mathscr{A}, E) $, denote the set of all inner $ (\upvarphi, \uppsi) $-derivations from $ \mathscr{A} $ to $ E $.
The first cohomology group $ H_{\upvarphi, \uppsi}^{1}(\mathscr{A}, E) $ is defined to be the quotient space
 $ \mathcal{Z}_{\upvarphi, \uppsi}^{1}(\mathscr{A}, E) / \mathcal{N}_{\upvarphi, \uppsi}^{1}(\mathscr{A}, E) $.
 In the case that $ \upvarphi= \uppsi $, we use the notations  $ \mathcal{Z}_{\upvarphi}^{1}(\mathscr{A}, E) $, $ \mathcal{N}_{\upvarphi}(\mathscr{A}, E) $ and $ H_{\upvarphi}^{1}(\mathscr{A}, E) $.


\begin{theorem} \label{1}
Let $ \mathscr{A}, \mathscr{B} $ be Banach algebras, $ \upvarphi \in \mathrm{Hom}(\mathscr{B}, \mathscr{A}) $,
 $ \mathscr{C} = \varphi(\mathscr{B}) $ is closed and
$ \mathscr{H}^{1}_{\upvarphi}(\mathscr{B}, \mathscr{C}^{*}) = \{ 0 \} $. Then
\begin{itemize}
\item[(i)] $ \varphi(\mathscr{B}) = \overline{\varphi(\mathscr{B}^{2})} $;
\item[(ii)] If $ \mathscr{A} $ is commutative and $ \upvarphi(\mathscr{B}) = \mathscr{A} $, then
$ \mathcal{Z}^{1}(\mathscr{A}, E) = \{ 0\} $ for each Banach $ \mathscr{A} $-bimodule $ E $ with following condition
 \begin{equation}\label{12}
a.x \neq 0 \hspace{.15cm} \Longleftrightarrow \hspace{.15cm} x.a \neq 0, \hspace{.5cm} (a \in \mathscr{A}, x \in E).
\end{equation}
\end{itemize} 
\end{theorem}

 \begin{proof} (i) Asuume that $ \upvarphi(\mathscr{B}) \setminus \overline{\upvarphi(\mathscr{B}^{2})} \neq \emptyset $. Take $ \upvarphi(b_{0}) \in \upvarphi(\mathscr{B}) \setminus \overline{\upvarphi(\mathscr{B}^{2})} $, choose $ \lambda_{0} \in \mathscr{C}^{*} $ with  $ \lambda_{0}|_{\overline{\varphi(\mathscr{B}^{2})}} \equiv 0 $
 and $ \langle \upvarphi(b_{0}), \lambda_{0} \rangle = 1 $. Define
 \begin{align*}
 D & = \lambda_{0} \otimes \lambda_{0}: \mathscr{B} \hookrightarrow \mathscr{C}^{*} \\
     & b \longmapsto D(b) := \langle \upvarphi(b_{0}), \lambda_{0} \rangle \lambda_{0}
 \end{align*}
 Certainly $ D $ is a continuous linear map and since $ \lambda_{0}|_{\overline{\varphi(\mathscr{B}^{2})}} \equiv 0 $, we have
 $ D(b_{1}b_{2}) = \langle \upvarphi(b_{1}b_{2}), \lambda_{0} \rangle \lambda_{0} = 0 $ ( $ b_{1}, b_{2} \in \mathscr{B} $). So
 \begin{align*}
 \big\langle \upvarphi(b), \upvarphi(b_{1})D(b_{2}) &+D(b_{1})\upvarphi(b_{2}) \big\rangle = \big\langle \upvarphi(b), \upvarphi(b_{1})D(b_{2}) \big\rangle  + \big\langle \upvarphi(b),  D(b_{1})\upvarphi(b_{2}) \big\rangle \\
 & = \big\langle \upvarphi(b)\upvarphi(b_{1}), D(b_{2}) \big\rangle  + \big\langle \upvarphi(b_{2})\upvarphi(b),  D(b_{1}) \big\rangle \\
   & = \big\langle \upvarphi(bb_{1}), \big\langle \upvarphi(b_{2}), \lambda_{0} \big\rangle \lambda_{0} \big\rangle  + \big\langle \upvarphi(b_{2}b),  \big\langle \upvarphi(b_{1}), \lambda_{0} \big\rangle \lambda_{0} \big\rangle \\
    & = \big\langle \upvarphi(bb_{1}), \lambda_{0} \big\rangle \big\langle \upvarphi(b_{2}), \lambda_{0} \big\rangle  + \big\langle \upvarphi(b_{2}b), \lambda_{0} \big\rangle \big\langle \upvarphi(b_{1}), \lambda_{0} \big\rangle = 0.
 \end{align*}
Therefore $ D(b_{1}b_{2}) = \upvarphi(b_{1})D(b_{2}) + D(b_{1})\upvarphi(b_{2}) $ and so $ D \in \mathcal{Z}^{1}_{\upvarphi}(\mathscr{B}, \mathscr{C}^{*})  $.  Now we have
\begin{align*}
   \langle \upvarphi(b_{0}), D(b_{0}) \big\rangle  &= \big\langle \upvarphi(b_{0}), \big\langle \upvarphi(b_{0}), \lambda_{0} \big\rangle \lambda_{0} \big\rangle    \\
         &  = \big\langle \upvarphi(b_{0}), \lambda_{0} \big\rangle  \big\langle \upvarphi(b_{0}), \lambda_{0} \big\rangle = 1
\end{align*}
But for $ \delta_{\lambda}(b) = [\lambda, \upvarphi(b_{0})] $, ($ \lambda \in \mathscr{C}^{*} $) we have
\begin{align*}
 \langle \upvarphi(b_{0}), \delta_{\lambda}(b_{0}) \big\rangle  &= \big\langle \upvarphi(b_{0}), \lambda\upvarphi(b_{0}) - \upvarphi(b_{0})\lambda \big\rangle \\
   & = \big\langle \upvarphi(b_{0}), \lambda\upvarphi(b_{0}) \big\rangle - \big\langle \upvarphi(b_{0}), \upvarphi(b_{0})\lambda \big\rangle \\
   & = \big\langle \upvarphi(b_{0}^{2}), \lambda \big\rangle - \big\langle \upvarphi(b_{0}^{2}), \lambda \big\rangle = 0, \hspace{.5cm} (\lambda \in \mathscr{C}^{*})
\end{align*}
Therefore $ D $ is not $ \upvarphi $-inner derivation, but  $ \mathscr{H}^{1}_{\upvarphi}(\mathscr{B}, \mathscr{C}^{*}) = \{0\} $.

(ii)  Assume that $ D\in \mathcal{Z}^{1}(\mathscr{A}, E) $ with $ D \neq 0 $, the partition (i) and equality $ \upvarphi(\mathscr{B}) = \mathscr{A} $, implies that $ \overline{\upvarphi(\mathscr{B}^{2})} = \upvarphi(\mathscr{B}) $
 and $ D' = D\circ \upvarphi \in \mathcal{Z}_{\upvarphi}^{1}(\mathscr{B}, E) $. So there is a $ b_{0} \in \mathscr{B} $ such that
 $$
 D'(b_{0}^{2}) = D\circ \upvarphi(b_{0}^{2}) \neq 0.
 $$
 Hence $ \upvarphi(b_{0}).D'(b_{0}) + D'(b_{0}).\upvarphi(b_{0}) \neq 0 $
 It follows by  condition (\ref{12}) that $ \upvarphi(b_{0}).D'(b_{0}) \neq 0 $. Put $ N = \overline{L.S\{a.x - x.a | a\in \mathscr{A}, x\in E \}^{\|\cdot\|}} $. Since $ N $ is a closed linear subspace of $ E $ and $ \upvarphi(b_{0}).D'(b_{0}) \notin N $
there is a $ \lambda \in E^{*} $ such that $ \lambda|_{N}\equiv 0 $ and $ \lambda(\upvarphi(b_{0}).D'(b_{0})) \neq 0 $. It follows by ([1], Proposition 2.6.6) that there is a $ R_{\lambda} \in B_{\mathscr{A}}(E, \mathscr{A}^{*}) $ such that
$ \big\langle a, R_{\lambda}(x)\big\rangle = \big\langle a.x, \lambda\big\rangle $. Clearly we have
$ R_{\lambda}(x).c = R_{\lambda}(x.c) = R_{\lambda}(c.x) = c.R_{\lambda}(x) $, ($ c \in \mathscr{A}, x \in E $). Indeed for each $ a, c \in \mathscr{A} $ and $ x\in E $ we have
\begin{align*}
\big\langle a, R_{\lambda}(x.c)\big\rangle &= \big\langle a.(x.c), \lambda \big\rangle = \big\langle (x.c).a, \lambda\big\rangle = \big\langle x.(ca), \lambda\big\rangle \\
   & = \big\langle x.(ac), \lambda\big\rangle  \big\langle (ac).x, \lambda\big\rangle = \big\langle a.(c.x), \lambda\big\rangle = \big\langle a, R_{\lambda}(c.x)\big\rangle
\end{align*}
On the other hands
\begin{align*}
 \big\langle a, R_{\lambda}(c.x)\big\rangle &= \big\langle (ac).x, \lambda\big\rangle \\
   &= \big\langle ac, R_{\lambda}(x)\big\rangle = \big\langle a, c.R_{\lambda}(x)\big\rangle
\end{align*}
And
\begin{align*}
 \big\langle a, R_{\lambda}(x.c)\big\rangle &= \big\langle x.(ca), \lambda\big\rangle = \big\langle (ca).x, \lambda\big\rangle  \\
                   &= \big\langle ca, R_{\lambda}(x)\big\rangle = \big\langle a, R_{\lambda}(x).c\big\rangle
\end{align*}
Therefore $ R_{\lambda}\circ D' \in \mathcal{Z}_{\upvarphi}^{1}(\mathscr{B}, \mathscr{A}^{*}) $. It follows by assumption
$ \mathscr{H}^{1}_{\upvarphi}(\mathscr{B}, \mathscr{A}^{*}) = \{ 0 \} $ that there is a $ \mu \in \mathscr{A}^{*} $ such that $ R_{\lambda}\circ D'(b) = [\mu, \upvarphi(b)] $ (i.e. $ R_{\lambda}\circ D' $ is a $ \upvarphi $-inner derivation) and we have
\begin{align*}
 1 = \big\langle \upvarphi(b_{0}).D'(b_{0}), \lambda \big\rangle &= \big\langle \upvarphi(b_{0}), R_{\lambda}\circ D'(b_{0})\big\rangle \\
   &= \big\langle \upvarphi(b_{0}), \mu.\upvarphi(b_{0}) - \upvarphi(b_{0}).\mu \big\rangle \\
   &= \big\langle \upvarphi(b_{0}), \mu.\upvarphi(b_{0}) \big\rangle  - \big\langle \upvarphi(b_{0}), \upvarphi(b_{0}).\mu \big\rangle \\
 &= \big\langle \upvarphi(b_{0}^{2}), \mu \big\rangle  - \big\langle \upvarphi(b_{0}^{2}), \mu \big\rangle = 0.
\end{align*}
Which is impossible.
 \end{proof}


\begin{remark}
Let in condition $ (ii) $ of Theorem \ref{2}, 
$ \mathscr{A} $ be a subalgebra of $ \mathscr{B} $ and we have $ \varphi^{2} = \varphi $. then by changing condition 
 $ \mathscr{H}^{1}_{\varphi}(\mathscr{B}, \mathscr{A}^{*}) = \{ 0 \} $
with 
 $ \mathscr{H}^{1}_{\varphi}(\mathscr{A}, \mathscr{A}^{*}) = \{ 0 \} $,
the proof is still valid.
\end{remark}


\begin{remark} 
Let $ \mathscr{B} $ be a unital $ C^{*} $-algebra, 
  $ a\in \mathscr{B}_{sa} $ and let 
 $ \mathscr{A} = \mathfrak{A}^{*}(a, 1_{\mathscr{B}}) $ 
 be the $ C^{*} $-algebra generated by 
  $ \{ a, 1_{\mathscr{B}} \} $ and let  $ E $
be a Banach $ \mathscr{B} $-bimodule. 
If 
  $ \mathcal{R}:\mathscr{B}\rightarrow \mathscr{A} $
 be the restriction mapp
  \footnote{the Restriction mapp from $ \mathscr{B} $ onto $ \mathscr{A} $. i.e.
  $ \mathcal{R}^{2} = \mathcal{R}, ~~ \mathcal{R}(a) = a, ~(a\in \mathcal{A}) $.
  }
 from $ \mathscr{B} $ to $ \mathscr{A} $, then 
  $ E $ is a Banach $ \mathscr{A} $-bimodule with the following actions 
  $$
  a.x = \mathcal{R}(a).x, ~~~~ x.a = x.\mathcal{R}(a)~~~~~(a\in \mathscr{A}, x\in E)
  $$
  amd for each 
  $ \delta \in \mathcal{Z}^{1}(\mathscr{B}, E) $, 
  we have
 $$
 \delta|_{\mathscr{A}} = \delta\circ \mathcal{R} \in \mathcal{Z}_{\mathcal{R}}^{1}(\mathscr{B}, E) ~~\Rightarrow~~ \delta\circ \mathcal{R} \in \mathcal{Z}_{\mathcal{R}}^{1}(\mathscr{A}, E) = \mathcal{Z}^{1}(\mathscr{A}, E).
 $$
Now since every  $ C^{*} $-algebra is weakly amenable 
 \footnote{Each $ C^{*} $-algebra $ \mathscr{A} $ is weakly amenable. i.e.
 $ \mathscr{H}^{1}(\mathscr{A}, \mathscr{A}^{*}) = \{ 0 \} $.
 }
 we have 
 $$
 \mathscr{H}^{1}_{\mathcal{R}}(\mathscr{A}, \mathscr{A}^{*}) = \mathscr{H}^{1}(\mathscr{A}, \mathscr{A}^{*}) = \{ 0 \}.
 $$
 On the other hand for any 
 $ b\in \mathscr{B} $, 
 $$
 \sigma(ab)\cup \{0\} = \sigma(ba)\cup \{0\}
 $$
 so 
 $ (ab \neq 0 \Leftrightarrow ba \neq 0) $.
 If $ \delta:\mathscr{B}\rightarrow \mathscr{B} $ be a derivation, then it follows by \ref{1}
that $ \delta\circ \mathcal{R} = 0 $, 
 therefore  $ \delta \equiv 0 $. 
So there is no every where defined nonzero derivation on $ C^{*} $-algebras. 
\end{remark}


\begin{theorem} \label{2}
Let $ \mathscr{M} $ be a Banach algebra which has a separating family of normal tracial states and let $ \mathscr{A} $ be a weakly amenable Banach subalgebra of  $ \mathscr{M} $ and let  $ \mathscr{M}^{2} = \mathscr{M} $. Then each derivation $ \delta $ from $ \mathscr{A} $ to $ \mathscr{M} $ is zero.
\end{theorem}

\begin{proof}
Let $ \delta $ be a non-zero derivation from $ \mathscr{A} $ to $ \mathscr{M} $. Since $ \mathscr{M}^{2} = \mathscr{M}$
 (\cite{Dales},  2.8.63), there is an element $ x \in \mathscr{M} $ such that $ \delta(x^{2}) \neq 0 $. It follows that $ x.\delta(x) +\delta(x).x \neq 0 $. Since there is a separating family of normal tracial states on $ \mathscr{M} $, so for some normal tracial state $ \tau $ on $ \mathscr{M} $ we have
$$
 2\tau(x.\delta(x)) = \tau(x.\delta(x) +\delta(x).x) \neq 0.
 $$
Therefore $ \tau(x.\delta(x)) \neq 0 $. Define $ R_{\tau}: \mathscr{M} \rightarrow \mathscr{A}^{*} $ as follows
 $$
 \big\langle a, R_{\tau}(x)\big\rangle = \big\langle a.x, \tau \big\rangle, \hspace{.5cm} (a \in \mathscr{A}, x \in \mathscr{M}).
 $$
Clearly $ R_{\tau} $ is $ \mathscr{A} $-bimodule map (i.e. $ R_{\tau}(a.x) = a.R_{\tau}(x) $ and $ R_{\tau}(x.a) = R_{\tau}(x).a $ holds for each $ a\in \mathscr{A}, x \in \mathscr{M} $).
If $ x_{\alpha} \longrightarrow x $, then
\begin{align*}
 \big\langle a, R_{\tau}(x)\big\rangle &= \big\langle a.x, \tau \big\rangle = \big\langle \lim_{\alpha} a.x_{\alpha}, \tau \big\rangle \\
   & = \lim_{\alpha}\big\langle a.x_{\alpha}, \tau \big\rangle = \lim_{\alpha}\big\langle a, R_{\tau}(x_{\alpha})\big\rangle  =  \big\langle a,   \lim_{\alpha} R_{\tau}(x_{\alpha})\big\rangle
\end{align*}
It follows by closed graph theorem that $ R_{\tau} $ is bounded. And so $ D = R_{\tau}\circ \delta \in \mathcal{Z}^{1}(\mathscr{A}, \mathscr{A}^{*}) $. Indeed
\begin{align*}
D(ab) = R_{\tau}\circ \delta(ab) &= R_{\tau}(\delta(ab)) =  R_{\tau}(a.\delta(b)+ \delta(a).b) \\
      &=  R_{\tau}(a.\delta(b)) + R_{\tau}(\delta(a).b) =  a.R_{\tau}(\delta(b)) + R_{\tau}(\delta(a)).b  \\
      &=  a.R_{\tau}\circ \delta(b) + R_{\tau}\circ \delta(a).b = a.D(b) + D(a).b   \\
\end{align*}
Now since $ \mathscr{H}^{1}(\mathscr{A}, \mathscr{A}^{*}) = \{ 0 \} $ (each $ C^{*} $-algebra is weakly amenable), there is a $ \lambda \in \mathscr{A}^{*} $
such that $ R_{\tau}\circ \delta(x) = [\lambda, x] $ and we have
\begin{align*}
0 \neq  \big\langle x.\delta(x), \tau \big\rangle &=  \big\langle x, R_{\tau}\circ \delta(x)\big\rangle =  \big\langle x, \lambda.x-x.\lambda\big\rangle \\
    &=  \big\langle x, \lambda.x\big\rangle - \big\langle x, x.\lambda\big\rangle = \big\langle x^{2}, \lambda\big\rangle - \big\langle x^{2}, \lambda\big\rangle = 0.
\end{align*}
This is a contradiction. So $ \delta\equiv 0 $.
\end{proof}


\section{Jordan $ (\upvarphi, \uppsi) $-Derivations on $ C^{*} $-algebras}

Let $ \upvarphi \in \mathscr{H}^{1}(\mathscr{A}, \mathscr{B}) $ and $ \uppsi \in \mathscr{H}^{1}(\mathscr{A}, \mathscr{C}) $. A linear map $ \delta $ from a Banach algebra $ \mathscr{A} $ to a Banach $ \mathscr{B}-\mathscr{C} $-bimodule as called a Jordan $ (\upvarphi, \uppsi) $-derivation proved that $ \delta(a^{2}) = \upvarphi(a)\delta(a)+\delta(a)\uppsi(a) $
for each $ a \in \mathscr{A} $. Clearly $ (\upvarphi, \uppsi) $-derivations are Jordan $ (\upvarphi, \uppsi) $-derivations. Using the fact that
$ ab + ba = (a+b)^{2}-a^{2}-b^{2} $.
It is easy to proved that the Joedan Jordan $ (\upvarphi, \uppsi) $-derivation condition is equivalent to
$$
\delta(ab+ba) = \upvarphi(a).\delta(b)+\delta(b).\uppsi(a) + \delta(a).\uppsi(b)+ \upvarphi(b).\delta(a).
$$


\begin{proposition} \label{31} 
Every bounded Jordan $ (\upvarphi, \uppsi) $-derivation $ \delta $ from a  von Neumann algebra $ \mathscr{M} $ to a unital Banach $ \mathscr{B}-\mathscr{C} $-bimodule $ X $ is a $ (\upvarphi, \uppsi) $-derivation.
\end{proposition}

\begin{proof}
Clearly $ X $ is a Banach $ \mathscr{M} $-bimodule with the following actions
$$
m~ \smallblacktriangleright ~x : = \upvarphi(m).x , \hspace{.5cm} x~ \smallblacktriangleleft~ m : = x.\uppsi(m), \hspace{.5cm}(m \in \mathscr{M}, x \in X)
$$
Therefore $ \delta $ is a bounded Jordan derivation from $ \mathscr{M} $ to a unital Banach $ \mathscr{M} $-bimodule $ X $. It follows from ([3], Lemma 2.1) that $ \delta $  is a derivation from $ \mathscr{M} $ to $ X $. Equivalently  $ \delta $ is a bounded linear mapping which satisfy the following condition
$$
\delta(ab) = a \smallblacktriangleright \delta(b) + \delta(a) \smallblacktriangleleft b, \hspace{.5cm}(a,b \in \mathscr{M})
$$
So $ \delta $ is a bounded linear mapping which satisfies the following condition
$$
\delta(ab) = \upvarphi(a).\delta(b) + \delta(a).\uppsi(b), \hspace{.5cm}(a,b \in \mathscr{M})
$$
Hence $ \delta $ as a bounded $ (\upvarphi, \uppsi) $-derivation
\end{proof}



\begin{lemma}(\textbf{Main One}) \label{32}
For unital Banach algebra $ \mathscr{A} $, the following asserations are equivalent :
\begin{itemize}
\item[(i)]  Every bounded Jordan $ (\upvarphi, \uppsi) $-derivation from $ \mathscr{A} $ to any unital Banach $ \mathscr{B}-\mathscr{C} $-bimodule is a $ (\upvarphi, \uppsi) $-derivation.
\item[(ii)]  Every bounded trilinear form $ V:\mathscr{A}\times\mathscr{C}\times\mathscr{B}\rightarrow \mathbb{C} $ which satisfies
\begin{equation}
V(a^{2},c,b) = V(a, \uppsi(a)c, b) +V(a,c,b\upvarphi(a)), \hspace{.5cm}(a\in \mathscr{A}, b \in \mathscr{B}, c\in \mathscr{C})
\end{equation}
Will also satisfy
\begin{equation}
V(ad,c,b) = V(a, \uppsi(d)c, b) +V(d,c,b\upvarphi(a)), \hspace{.5cm}(a,d\in \mathscr{A}, b \in \mathscr{B}, c\in \mathscr{C})
\end{equation}
\end{itemize}
\end{lemma}

\begin{proof}
$ (i)\Rightarrow (ii) $

The projective tensor product $ \mathscr{C}\widehat{\otimes}\mathscr{B} $
is a unital $ \mathscr{C}-\mathscr{B} $-bimodule with the following actions
$$
c_{0}.(c\otimes b) : = (c_{0}c)\otimes b, \hspace{.5cm} (c\otimes b).b_{0} : = c\otimes(bb_{0}),\hspace{.5cm} (b,b_{0} \in \mathscr{B}, c,c_{0} \in \mathscr{C})
$$
And so $ (\mathscr{C}\widehat{\otimes}\mathscr{B})^{*} $ is a unital $ \mathscr{B}-\mathscr{C} $-bimodule with the following actions
\begin{align*}
\big\langle c\otimes b, \lambda.c_{0} \big\rangle &= \big\langle c_{0}.(c\otimes b), \lambda \big\rangle \\
\big\langle c\otimes b, b_{0}.\lambda \big\rangle &= \big\langle (c\otimes b).b_{0}, \lambda \big\rangle, \hspace{.5cm}(\lambda \in (\mathscr{C}\widehat{\otimes}\mathscr{B})^{*}), b,b_{0} \in \mathscr{B}, c,c_{0} \in \mathscr{C})
\end{align*}
To each bounded trilinear map $ V:\mathscr{A}\times\mathscr{C}\times\mathscr{B}\rightarrow \mathbb{C} $ which satisfies the relation (3.1), we associate a bounded linear map $ \delta: \mathscr{A} \rightarrow (\mathscr{C}\widehat{\otimes}\mathscr{B})^{*} $ by the following definition
$$
\big\langle c\otimes b, \delta(a) \big\rangle := V(a, b, c ), \hspace{.5cm}(a\in \mathscr{A}, b \in \mathscr{B}, c\in \mathscr{C})
$$
Now we show that $ \delta $ is a bounded Jordan $ (\upvarphi, \uppsi) $-derivation.
\begin{align*}
\big\langle c\otimes b, \delta(a^{2}) \big\rangle &= V(a^{2}, b, c ) = V(a, \uppsi(a)c, b) +V(a,c,b\upvarphi(a)) \\
      & = \big\langle (\uppsi(a)c)\otimes b, \delta(a) \big\rangle + \big\langle c\otimes (b\upvarphi(a)), \delta(a) \big\rangle \\
      & = \big\langle \uppsi(a).(c\otimes b), \delta(a) \big\rangle + \big\langle (c\otimes b).\upvarphi(a), \delta(a) \big\rangle \\
      & = \big\langle c\otimes b, \delta(a).\uppsi(a) \big\rangle + \big\langle c\otimes b, \upvarphi(a).\delta(a) \big\rangle \\
      & = \big\langle c\otimes b, \delta(a).\uppsi(a) + \upvarphi(a).\delta(a) \big\rangle, \hspace{.5cm}(a\in \mathscr{A}, b \in \mathscr{B}, c\in \mathscr{C}) \\
\end{align*}
Therefore $  \delta(a^{2}) = \delta(a).\uppsi(a) + \upvarphi(a).\delta(a)  $, ($ a\in \mathscr{A} $). So  $ \delta $ is a bounded Jordan $ (\upvarphi, \uppsi) $-derivation and by hypothesis, $ \delta $ is a  $ (\upvarphi, \uppsi) $-derivation and we have
 \begin{align*}
V(ad,c,b) &= \big\langle c\otimes b, \delta(ad) \big\rangle \hspace{.5cm}(a,d\in \mathscr{A}, b \in \mathscr{B}, c\in \mathscr{C}) \\
     &= \big\langle c\otimes b, \delta(a).\uppsi(d) + \upvarphi(a).\delta(d) \big\rangle \\
      &=  \big\langle c\otimes b, \delta(a).\uppsi(d)  \big\rangle + \big\langle c\otimes b, \upvarphi(a).\delta(d) \big\rangle \\
    &=  \big\langle \uppsi(d).(c\otimes b), \delta(a)  \big\rangle + \big\langle (c\otimes b). \upvarphi(a),\delta(d) \big\rangle \\
    &=  \big\langle (\uppsi(d)c)\otimes b, \delta(a)  \big\rangle + \big\langle c\otimes (b\upvarphi(a)),\delta(d) \big\rangle \\
    &= V(a, \uppsi(d)c, b) +V(d,c,b\upvarphi(a))
\end{align*}
$ (ii)\Rightarrow (i) $
 Let $ E $ be a unital Banach $ \mathscr{B}-\mathscr{C} $-bimodule and $ \delta: \mathscr{A} \rightarrow E $ be a bounded Jordan $ (\upvarphi, \uppsi) $-derivation.
 
  We associate to each $ \sigma \in E^{*} $, the bounded trilinear form
$ V_{\sigma}: \mathscr{A}\times\mathscr{C}\times\mathscr{B}\rightarrow \mathbb{C} $ given by
$ V_{\sigma}(a,c,b) := \big\langle b.\delta(a).c, \sigma \big\rangle $, $ (a,d\in \mathscr{A}, b \in \mathscr{B}, c\in \mathscr{C}) $ and we have
\begin{align*}
V_{\sigma}(a^{2},c,b) &= \big\langle b.\delta(a^{2}).c, \sigma \big\rangle \\
   &= \big\langle b.\big(\delta(a).\uppsi(a) + \upvarphi(a)\delta(a)\big).c, \sigma \big\rangle \\
   &=  \big\langle b.\big(\delta(a).\uppsi(a)\big).c, \sigma \big\rangle+\big\langle b.\big(\upvarphi(a).\delta(a)\big).c, \sigma \big\rangle \\
   &=  \big\langle b.\delta(a).\big(\uppsi(a)c\big), \sigma \big\rangle+\big\langle \big(b\upvarphi(a)\big).\delta(a).c, \sigma \big\rangle \\
   &= V_{\sigma}(a,\uppsi(a)c,b) +  V_{\sigma}(a,c,b\upvarphi(a))
\end{align*}
Therefore $ V_{\sigma} $ satisfy the condition (3.2), and so it satisfy the condition (3.3) and we have
\begin{align*}
V_{\sigma}(a^{2},c,b) &= \big\langle b.\delta(a^{2}).c, \sigma \big\rangle \\
   &= \big\langle b.\big(\delta(a).\uppsi(a) + \upvarphi(a)\delta(a)\big).c, \sigma \big\rangle \\
   &=  \big\langle b.\big(\delta(a).\uppsi(a)\big).c, \sigma \big\rangle+\big\langle b.\big(\upvarphi(a).\delta(a)\big).c, \sigma \big\rangle \\
   &=  \big\langle b.\delta(a).\big(\uppsi(a)c\big), \sigma \big\rangle+\big\langle \big(b\upvarphi(a)\big).\delta(a).c, \sigma \big\rangle \\
   &= V_{\sigma}(a,\uppsi(a)c,b) +  V_{\sigma}(a,c,b\upvarphi(a))
\end{align*}
Therefore $ V_{\sigma} $ ($ \sigma\in E^{*} $) satisfy the condition (3.2), and so it satisfy the condition (3.3) and we have
\begin{align*}
\big\langle \delta(ad), \sigma \big\rangle &= \big\langle 1_{\mathscr{B}}.\delta(ad).1_{\mathscr{C}}, \sigma \big\rangle = V_{\sigma}(ad,1_{\mathscr{C}},1_{\mathscr{B}}) \\
       &= V_{\sigma}(a,\uppsi(d).1_{\mathscr{C}},1_{\mathscr{B}}) +  V_{\sigma}(d,1_{\mathscr{C}},1_{\mathscr{B}}.\upvarphi(a)) \\
        &= \big\langle 1_{\mathscr{B}}.\delta(a).\uppsi(d).1_{\mathscr{C}}, \sigma \big\rangle +  \big\langle 1_{\mathscr{B}}.\upvarphi(a).\delta(d).1_{\mathscr{C}}, \sigma \big\rangle \\
      &= \big\langle \delta(a).\uppsi(d) + \upvarphi(a).\delta(d), \sigma \big\rangle, \hspace{.5cm}(\sigma \in E^{*})
\end{align*}
Hence $ \delta $ is a $ (\upvarphi, \uppsi) $-derivation.
\end{proof}


\begin{theorem} 
Let $ \mathscr{A} $ be a $ C^{*} $-algebra. Every bounded Jordan $ (\upvarphi, \uppsi) $-derivation $ \delta $ from $ \mathscr{A} $ to a Banach $ \mathscr{B}-\mathscr{C} $-bimodule $ X $ is a $ (\upvarphi, \uppsi) $-derivation.
\end{theorem}

\begin{proof}
It suffix to change  (Proposition 3.1  and  Lemma 3.2) resoectively  with (Proposition 2.2  and  Lemma 2.3)   at the proof of Theorem 2.4 in \cite{Haagerup}.
\end{proof}


\section{(Approximately) 2-Local $ (\upvarphi, \uppsi) $-Derivations}

\textbf{Preliminaries}

\begin{definition}
A mapping $ \Delta $ from a Banach algebra $ \mathscr{A} $ into a Banach $ \mathscr{B}-\mathscr{C} $-bimodule $ E $
is bounded 2-local (respectively, approximately 2-local) $ (\upvarphi, \uppsi) $-derivation, if for each $ a,b \in \mathscr{A} $, there is a bounded $ (\upvarphi, \uppsi) $-derivation $ D_{a,b} $ (respectively, a sequence of bounded $ (\upvarphi, \uppsi) $-derivations $ \{ D^{n}_{a,b} \} $) from $ \mathscr{A} $ into $ E $ such that $ D(a) = D_{a,b}(a) $ and $ D(b) = D_{a,b}(b) $ (respectively, $ D(a) = \lim_{n\rightarrow\infty} D^{n}_{a,b}(a) $ and $ D(b) = \lim_{n\rightarrow\infty} D^{n}_{a,b}(b) $).
\end{definition}


\begin{lemma} \label{42}
Let $ \Delta $ be a 2-local (or an approximately 2-local) $ (\upvarphi, \uppsi) $-derivation of a Banach algebra $ \mathscr{A} $ into Banach $ \mathscr{B}-\mathscr{C} $-bimodule $ E $. Then
\begin{itemize}
\item[(i)]  $ \Delta(\lambda x) = \lambda \Delta(x) $ for any $ \lambda \in \mathbb{C} $ and $ x \in \mathscr{A} $;
\item[(ii)]  $ \Delta(x^{2}) = \Delta(x).\uppsi(x) + \upvarphi(x).\Delta(x) $ for any $ x \in \mathscr{A} $.
\end{itemize}
\end{lemma}

\begin{proof}
We prove ths lemma only for approxmately 2-local  $ (\upvarphi, \uppsi) $-derivations o Baach algebras.
\begin{itemize}
\item[(i)]  For each $ x\in \mathscr{A} $ ad $ \lambda\i \mathbb{C} $, there iexists a sequece of $ (\upvarphi, \uppsi) $-derivations
$ \{ D^{n}_{x,\lambda x} \} $ such that
\begin{align*}
\Delta(x) &=  \lim_{n\rightarrow\infty} D^{n}_{x,\lambda x}(x); \\
 \Delta(\lambda x) &=  \lim_{n\rightarrow\infty} D^{n}_{x, \lambda x}( \lambda x) = \lambda \Delta(x)
\end{align*}
Hece $ \Delta $ is homogeneous.
\item[(ii)]  For each $ x\in \mathscr{A} $, there exists $ (\upvarphi, \uppsi) $-derivations
$ \{ D^{n}_{x,x^{2}} \} $ such that
\begin{align*}
\Delta(x) &=  \lim_{n\rightarrow\infty} D^{n}_{x,x^{2}}(x); \\
 \Delta(x^{2}) &=  \lim_{n\rightarrow\infty} D^{n}_{x, x^{2}}(x^{2}) \\
&=(\lim_{n\rightarrow\infty} D^{n}_{x, x^{2}}(x)).\uppsi(x) + \upvarphi(x).(\lim_{n\rightarrow\infty} D^{n}_{x, x^{2}}(x))\\
 &= \Delta(x).\uppsi(x) + \upvarphi(x)\Delta(x).
\end{align*}
\end{itemize}
\end{proof}


\begin{lemma} \label{43}
Any additive 2-local $ (\upvarphi, \uppsi) $-derivation $ \Delta $ from a $ C^{*} $-algebra $ \mathscr{A} $ to a Banach  $ \mathscr{B}-\mathscr{C} $-bimodule $ E $ is a $ (\upvarphi, \uppsi) $-derivation.
\end{lemma}

\begin{proof}
We conclude from  Lemma 3.6 that each additive 2-local $ (\upvarphi, \uppsi) $-derivation $ \Delta $ from a $ C^{*} $-algebra $ \mathscr{A} $ to a Banach  $ \mathscr{B}-\mathscr{C} $-bimodule $ E $ is linear and satisfies
$ \Delta(x^{2}) = \Delta(x).\uppsi(x) + \upvarphi(x)\Delta(x) $ for any $ x\in \mathscr{A} $, so $ \Delta $ is a $ (\upvarphi, \uppsi) $-derivation.
\end{proof}


\textbf{Main Theorems}

\begin{theorem} \label{44}
 Let  $ \mathscr{M} $ be a finite von Neumann algebra and let $ \upvarphi, \uppsi $ be continuous homomorphisms on  $ \mathscr{M} $. Then each $ (\upvarphi, \uppsi) $-derivation  on $ \mathscr{M} $ is  a $ (\upvarphi, \uppsi) $-inner derivation. i.e. there is an element $ a_{0} $  in $ \mathscr{M} $  such that $ \delta(a) =  a_{0}\cdot \uppsi(a) - \upvarphi(a)\cdot  a_{0} $ and $ \|a_{0}\| \leq \|\delta\| $.
\end{theorem}

\begin{proof}
Let $ \mathscr{M}^{u} $ be the group of all unitary elements in $ \mathscr{M} $.
For $ u \in \mathscr{M} $, put
$$
T_{u}(x) := \big( \upvarphi(u).x + \delta(u)\big).\uppsi(u)^{-1}, \hspace{.5cm}(x\in  \mathscr{M})
$$
 If $ u, v \in  \mathscr{M}^{u} $, then
\begin{align*}
T_{u}T_{v}(x) &= \bigg( \upvarphi(u).\big( \upvarphi(v).x + \delta(v)\big).\uppsi(v)^{-1} +\delta(u) \bigg).\uppsi(u)^{-1} \\
               &= \bigg( \big( \upvarphi(uv).x + \upvarphi(u).\delta(v)\big).\uppsi(v)^{-1} +\delta(u) \bigg).\uppsi(u)^{-1} \\
               &= \upvarphi(uv).x.\uppsi(uv)^{-1} + \upvarphi(u).\delta(v).\uppsi(uv)^{-1} +\delta(u).\uppsi(v).\uppsi(v)^{-1}.\uppsi(u)^{-1} \\
               &= \upvarphi(uv).x.\uppsi(uv)^{-1} + \big(\upvarphi(u).\delta(v) +\delta(u).\uppsi(v)\big).\uppsi(uv)^{-1} \\
               &= \big(\upvarphi(uv).x + \delta(uv)\big).\uppsi(uv)^{-1} = T_{uv}(x)
\end{align*}
Hence $ T_{u}T_{v} = T_{uv} $, ($ u, v \in  \mathscr{M}^{u} $). Let $ \Delta $ be the set of all non-empty $ \sigma(\mathscr{M}, \mathscr{M}_{*}) $-closed convex sets $ \mathscr{K} $ in $ \mathscr{M} $ satisfying the following conditions
$$
1.\hspace{.1cm} T_{u}(\mathscr{K}) \subseteq \mathscr{K}, \hspace{.5cm} 2.\hspace{.1cm} \sup_{x \in \mathscr{K}} \|x\| \leq \|\delta\|
$$
Since
\begin{align*}
\big\| T_{u}(0)\big\| &= \big\| \big( \upvarphi(u).0 + \delta(u)\big).\uppsi(u)^{-1}\big\|  \\
             &=  \big\| \big( \delta(u)\big).\uppsi(u)^{-1}\big\| \leq \big\| \delta \big\| \big\| u \big\| \big\| \uppsi(u)^{-1}\big\| = \big\| \delta \big\|
\end{align*}
Therefore $ \Delta $ is on-empty. Define an order in $ \Delta $ by the set inclusion. Let
$ \big( \mathscr{K}_{\alpha}\big)_{\alpha\in I} $ be linearly ordered decreasing subsets in $ \Delta $. Then $ \bigcap_{\alpha\in I}\mathscr{K}_{\alpha} \in \Delta $, because $ \mathscr{K}_{\alpha} (\alpha\in I)$ is compact (Arzela Ascoli).  Hence there is a minimal element $ \mathscr{K}_{0} $ in $ \Delta $ by Zorns lemma.

If $ a,b \in \mathscr{K}_{0} $, then $ a - b \in \mathscr{K}_{0} \setminus \mathscr{K}_{0} $
and for $ u \in \mathscr{M}^{u} $, we have
\begin{align*}
 \upvarphi(u).\big( a - b \big).\uppsi(u)^{-1} &= \upvarphi(u). a .\uppsi(u)^{-1} - \upvarphi(u). b.\uppsi(u)^{-1} \\
             &= \big( \upvarphi(u). a .\uppsi(u)^{-1} + \delta(u).\uppsi(u)^{-1}\big) - \big( \upvarphi(u). b .\uppsi(u)^{-1} + \delta(u).\uppsi(u)^{-1}\big)  \\
               &= \big( \upvarphi(u). a + \delta(u)\big).\uppsi(u)^{-1} - \big( \upvarphi(u). b + \delta(u)\big).\uppsi(u)^{-1} \\
              &= T_{u}(a) - T_{u}(b) \in \mathscr{K}_{0} \setminus \mathscr{K}_{0}
\end{align*}
Hence $ \mathscr{K}_{0} \setminus \mathscr{K}_{0} $ is invariant under the mapping
$$
\Phi^{u}: x \longmapsto \upvarphi(u).\big( a - b \big).\uppsi(u)^{-1}, \hspace{.5cm}\big( x \in \mathscr{M}\big)
$$
i.e. $ \Phi^{u}\big( \mathscr{K}_{0} \setminus \mathscr{K}_{0}\big) \subseteq\mathscr{K}_{0} \setminus \mathscr{K}_{0} $.
Since $ \mathscr{M} $ is a finite von Neumann algebra, there is a faithful family of normal tracial states $ \uptau $  on $ \mathscr{M} $ (\cite{Bing}, Theorem 6.3.10).
For each $ \tau \in \uptau $, define seminorm $ P_{\tau} $ as follows
 $$
 P_{\tau}(x) := \tau(x^{*}x)^{\frac{1}{2}}, \hspace{.5cm}(x \in \mathscr{M}).
 $$
 Let $ \lambda = \sup_{x\in \mathscr{K}_{0}}P_{\tau}(x) $ and let $ a, b \in \mathscr{K}_{0} $, then for an arbitrary positive number $ \varepsilon > 0 $, there is an element $ u \in \mathscr{M}^{u} $ with $ P_{\tau}(\frac{a+b}{2}) > \lambda - \varepsilon $. Indeed, if there is $ \varepsilon_{0} > 0 $ such that $ \forall u \in \mathscr{M}^{u} $,
 $ P_{\tau}(T_{u}(\frac{a+b}{2})) \leq \lambda - \varepsilon_{0} $ then
 $ P_{\tau}(\frac{a+b}{2}) = P_{\tau}(T_{id}(\frac{a+b}{2})) \leq \lambda - \varepsilon_{0} $,  which for $ a=b $ implies that
 $ P_{\tau}(\frac{a+a}{2}) \leq \lambda - \varepsilon_{0} $, and so
 $ \lambda = \sup_{a\in \mathscr{K}_{0}}P_{\tau}(a) \leq \lambda - \varepsilon_{0} $ which is impossible.
Since
$ P_{\tau}(T_{u}(a)) \leq \lambda, P_{\tau}(T_{u}(b)) \leq \lambda $ we have
\begin{align*}
\bigg( P_{\tau}(\frac{T_{u}(a)+T_{u}(b)}{2})) \bigg)^{2} &=  \frac{1}{4}\tau\big( T_{u}(a)^{*}T_{u}(a) + T_{u}(a)^{*}T_{u}(b) + T_{u}(b)^{*}T_{u}(a) + T_{u}(b)^{*}T_{u}(b)\big)  \\
  \bigg( P_{\tau}(\frac{T_{u}(a)-T_{u}(b)}{2})) \bigg)^{2} &=  \frac{1}{4}\tau\big( T_{u}(a)^{*}T_{u}(a) - T_{u}(a)^{*}T_{u}(b) - T_{u}(b)^{*}T_{u}(a) + T_{u}(b)^{*}T_{u}(b)\big)
\end{align*}
It follows from the last two equalities that
\begin{align*}
\bigg( P_{\tau}(\frac{T_{u}(a)+T_{u}(b)}{2})) \bigg)^{2} &+ \bigg( P_{\tau}(\frac{T_{u}(a)-T_{u}(b)}{2})) \bigg)^{2} \\
&= \frac{1}{2}\bigg( P_{\tau}(T_{u}(a))^{2} + P_{\tau}(T_{u}(b))^{2}\bigg)
 - \bigg(P_{\tau}\big(T_{u}(\frac{a+b}{2})\big)\bigg)^{2} \\
      &\leq \frac{1}{2}\big( \lambda^{2} + \lambda^{2}\big) - \big( \lambda - \varepsilon\big)^{2} = \big(2\lambda - \varepsilon \big) \varepsilon
\end{align*}
So
\begin{align*}
      &\frac{1}{4}\tau\bigg( \big( T_{u}(a) - T_{u}(b)\big)^{*}\big( T_{u}(a) - T_{u}(b)\big)\bigg) = 0,\hspace{.5cm} (\tau \in \uptau) \\
\Longleftrightarrow \hspace{.5cm} &  \tau\bigg( \big( \upvarphi(u)\big( a - b \big)\uppsi(u)^{-1} \big)^{*}\big( \upvarphi(u)\big( a - b \big)\uppsi(u)^{-1} \big)\bigg) = 0,\hspace{.5cm} (\tau \in \uptau)\\
\Longleftrightarrow \hspace{.5cm} &  \tau\bigg( \uppsi(u)\big( a - b \big)^{*}\upvarphi(u)^{-1} \upvarphi(u).\big( a - b \big)\uppsi(u)^{-1} \bigg) = 0,\hspace{.5cm} (\tau \in \uptau)\\
\Longleftrightarrow \hspace{.5cm} &  \tau\bigg( \uppsi(u)\big( a - b \big)^{*}\big( a - b \big)\uppsi(u)^{-1} \bigg) = 0,\hspace{.5cm} (\tau \in \uptau)\\
\Longleftrightarrow \hspace{.5cm} &  \tau\bigg( \big( a - b \big)^{*}\big( a - b \big)\uppsi(u)^{-1}\uppsi(u) \bigg) = 0,\hspace{.5cm} (\tau \in \uptau)\\
\Longleftrightarrow \hspace{.5cm} &  \tau\bigg( \big( a - b \big)^{*}\big( a - b \big)\bigg) = 0,\hspace{.5cm} (\tau \in \uptau) \hspace{.5cm} \Longleftrightarrow \hspace{.5cm} a-b = 0.
\end{align*}
Hence $ \mathscr{K}_{0} $ consists  of only one elemente $ a_{0} $. Since each element of $ \mathscr{M} $ is
 a finite linear combination of unitary elements in $ \mathscr{M} $,
\footnote{
If $ \mathscr{A} $ is a $ C^{*} $-algebra and $ a \in \mathscr{A} $ be such that $ \|a\| \leq 1 $, Then $ a = b + ic $ where $ b, c \in \mathscr{A}_{sa} $ are self-adjoint elements and given by
$$
b = \frac{1}{2}(a + a^{*}), \hspace{.5cm} c = \frac{1}{2i}(a - ia^{*}).
$$
We can decompose b and c as
$$
b = \frac{1}{2}(U_{b} + V_{b}), \hspace{.5cm} c = \frac{1}{2}(U_{c} + V_{c}).
$$
where $ U_{b}, V_{b}, U_{c}, V_{c} $ are unitary and given by
\begin{align*}
  U_{b} = b + i\sqrt{1-b^{2}}, & \hspace{.5cm} V_{b} = b - i\sqrt{1-b^{2}} \\
  U_{c} = c + i\sqrt{1-c^{2}}, & \hspace{.5cm} V_{c} = c - i\sqrt{1-c^{2}}
\end{align*}
} 
  we have
$$
 T_{u}(a_{0}) = \big( \upvarphi(u).a_{0} + \delta(u)\big).\uppsi(u)^{-1}= a_{0}.
 $$
Therefore
 $ \delta(x) =  a_{0}\cdot \uppsi(x) - \upvarphi(x)\cdot  a_{0} $, ($ x \in \mathscr{M} $). Clearly  $ a_{0} \in \mathscr{K}_{0} $ implies that $ \|a_{0}\| = \sup_{x \in \mathscr{K}_{0}} \|x\| \leq \|\delta\| $.
\end{proof}



\begin{definition} \label{45} 
A linear functional $ \tau:\mathscr{M} \rightarrow \mathbb{C} $ is called $ (\upvarphi, \uppsi) $-tracial, 
if  
$$
 \tau(\upvarphi(x)y) = \tau(y\uppsi(x)), \hspace{.5cm} (x,y \in \mathscr{M}).
$$
 A von Neumann algebra $ \mathscr{M} $ is called $ (\upvarphi, \uppsi) $-finite, if there exists a faithful family of normal $ (\upvarphi, \uppsi) $-tracial states $ \mathfrak{T} $ on $ \mathscr{M} $. 
\end{definition}


\begin{theorem} \label{46}
Let $ \mathscr{M} $ be a finite and $ (\upvarphi, \uppsi) $-finite von Neumann algebra.
Then each 2-local $ (\upvarphi, \uppsi) $-derivation $ \Delta $  on $ \mathscr{M} $ with condition $ \{ \big( \Delta(u+v)-\Delta(u)-\Delta(v)\big)^{*}\} \subseteq \uppsi(\mathscr{M}) $,  is a $ (\upvarphi, \uppsi) $-derivation.
\end{theorem}

\begin{proof}
Let $ \Delta $ be a 2-local $ (\upvarphi, \uppsi) $-derivation and let $ \mathfrak{T} $ be a faithful family of normal $ (\upvarphi, \uppsi) $-tracial states on $ \mathscr{M} $ and $ \tau \in \mathfrak{T} $. For each $ x,y \in \mathscr{M} $ there exists a $ (\upvarphi, \uppsi) $-derivation $ D_{x,y} $ on $ \mathscr{M} $ such that $ \Delta(x) = D_{x,y}(x) $ and $ \Delta(y) = D_{x,y}(y) $. It follows from theorem \ref{44} that $ D_{x,y} $ is $ (\upvarphi, \uppsi) $-inner, so there is an element $ m \in \mathscr{M} $ such that
$$
m\uppsi(xy) - \upvarphi(xy)m = D_{x,y}(xy) = D_{x,y}(x)\upvarphi(y) + \upvarphi(x)D_{x,y}(y),
$$
Therefore
$$
\tau\big( D_{x,y}(x)\uppsi(y) + \upvarphi(x)D_{x,y}(y)\big) = \tau\big( m\uppsi(xy) - \upvarphi(xy)m\big) = 0,
$$
So
$$
\tau\big(  D_{x,y}(x)\uppsi(y)\big) = -\tau\big(\upvarphi(x)D_{x,y}(y)\big)
$$
Based on the above analysis, the following equality can be obtained
$$
\tau\big( \Delta(x)\uppsi(y)\big) = -\tau\big(\upvarphi(x)\Delta(y)\big)
$$
For arbitrary $ u, v, w \in \mathscr{M} $, set $ x= u+ v, y = w $. So we conclude that
\begin{align*}
\tau\big(  \Delta(u+ v)\uppsi(w)\big) &= -\tau\big(\upvarphi(u+v) \Delta(w)\big) \\
   &= -\tau\big(\upvarphi(u) \Delta(w)\big) -\tau\big(\upvarphi(v) \Delta(w)\big) \\
   &= \tau\big(  \Delta(u)\uppsi(w)\big) +\tau\big(  \Delta(v)\uppsi(w)\big) \\
   &= \tau\big(  (\Delta(u)+ \Delta(v))\uppsi(w)\big). \hspace{.5cm}(\tau\in \mathfrak{T})
\end{align*}
Hence
$$
\tau\bigg( (\Delta(u+ v)-\Delta(u)-\Delta(v))\uppsi(w)\bigg) = 0, \hspace{.5cm}(\tau\in \mathfrak{T})
$$
It folows from assumption that for each $ u,v \in \mathscr{M} $, there is a $ w\in \mathscr{M} $ such that
$ \uppsi(w) = \big( \Delta(u+v)-\Delta(u)-\Delta(v)\big)^{*} $.
So
$$
\tau\bigg(\big(\Delta(u+ v)-\Delta(u)-\Delta(v)\big)\big(\Delta(u+v)-\Delta(u)-\Delta(v)\big)^{*}\bigg) = 0, \hspace{.5cm}(\tau\in \mathfrak{T})
$$
Now since the family $ \mathfrak{T} $ is faithful, we have
$$
\Delta(u+ v)-\Delta(u)-\Delta(v) = 0,
$$
So
$$
\Delta(u+ v) = \Delta(u) + \Delta(v).
$$
It follows that $ \Delta $ is an additive 2-local $ (\upvarphi, \uppsi) $-derivation, and Lemma 3.7 implies that $ \Delta $ is a bounded Jordan $ (\upvarphi, \uppsi) $-derivation.
\end{proof}


\begin{remark}
 The last theorem also hold, if we  replace the condition $ \{ \big( \Delta(u+v)-\Delta(u)-\Delta(v)\big)^{*}\} \subseteq \uppsi(\mathscr{M}) $, with $ \{ \big( \Delta(u+v)-\Delta(u)-\Delta(v)\big)^{*}\} \subseteq \upvarphi(\mathscr{M}) $ one.
\end{remark}


\begin{remark}
 The last theorem hold also for approxmately 2-local $ (\upvarphi, \uppsi) $-derivations, if in addition, $ (\upvarphi, \uppsi) $-tracial map $ \tau $ is normal.
\end{remark}


\bibliographystyle{amsplain}

\begin{thebibliography}{n}



\bibitem{Bing} Li Bingren, Introduction to operator algebras, World Sci., Singapore, 1992.



\bibitem{Dales} H. G. Dales, Banach algebras and automatic continuity, London Math. Society Monographs, Volume 24, Clarendon Press, Oxford, 2000. MR1816726 (2002e:46001)


\bibitem{Haagerup} U. Haagerup and N. J. Laustsen, \emph{Weak amenability of $ C^{*}$-algebras and a theorem of Goldstein},  In Banach algebras 97, (ed. E. Albrecht and M. Mathieu), Walter de Gruyter,  Berlin, 1998, 223-243.




\bibitem{Sakai} S. Sakai,  \emph{$ C^{*}$-Algebras and $ W^{*}$-Algebras} (Classics in Mathematics)-Springer (1997).


\bibitem{Sakai2}  S. Sakai, \emph{Operator algebras in dynamical systems}, Cambridge University Press, 1991










\end{thebibliography}

\end{document}